\documentclass[12pt,reqno]{amsart}
\usepackage{amsmath,amsthm,amssymb}
\usepackage{amssymb, amsmath}
\usepackage{url}
\usepackage[utf8]{inputenc}
\usepackage[colorlinks=true,urlcolor=blue,pdfstartview=FitV, linkcolor=blue,citecolor=blue,bookmarksopen, bookmarksnumbered={true}]{hyperref}
\usepackage[usenames]{color}
\usepackage{cite}

\theoremstyle{plain}

\newtheorem{theorem}{Theorem}[section]
\newtheorem{definition}{Definition}[section]
\newtheorem{example}{Example}[section]

\newtheorem{corollary}{Corollary}[section]
\newtheorem{remark}{Remark}[section]
\newtheorem{proposition}{Proposition}[section]

\def\.{\cdot}
\def\<{\left\langle}
\def\>{\right\rangle}
\def\({\left(}
\def\){\right)}

\renewcommand{\geq}{\geqslant}
\renewcommand{\phi}{\varphi}

\def\L{\mathbb\L}

\def\subset{\subseteq}

\def\epsilon{\varepsilon}

\def\virt-dim{\operatorname{virt-dim}}


\begin{document}
\title[On Topological Rough Groups]{On Topological Rough Groups}
\author[N. Alharbi, A. Altassan, H. Aydi, C. \"{O}zel] {Nof Alharbi $^{1}$, Alla Altassan $^{2}$, Hassen Aydi $^{3}$, Cenap \"{O}zel $^{4}$}

\thanks{nof20081900@hotmail.com  $^{1}$, aaltassan@kau.edu.sa $^{2}$, hmaydi@uod.edu.sa $^{3}$, hassen.aydi@isima.rnu.tn $^{3}$, cenap.ozel@gmail.com $^{4}$}
\maketitle
\begin{center}
{\footnotesize $^{1,2,4}$ Department of Mathematics, King Abdulaziz University,
P.O.Box: 80203 Jeddah 21589, Saudi Arabia.}\\
{\footnotesize $^{3}$ Imam Abdulrahman Bin Faisal University,
Department of Mathematics. College of Education of Jubail,  P.O: 12020, Industrial Jubail 31961. Saudi Arabia.}
\end{center}

\begin{abstract}
In this paper, we give an introduction for rough groups and rough homomorphisms. Then we present some properties related to topological rough subgroups and rough subsets.  We construct the product of topological rough groups and give an illustrated example. Then, we define topological rough group homomorphisms and topological rough group homeomorphisms. Finally, we introduce a rough action, a rough homogenous space and a rough kernel.
\end{abstract}

\maketitle

\section{Introduction}
  In \cite{Abdull}, Bagirmaz et al. introduced  the concept of topological rough groups. They extended the notion of a topological group to include algebraic structures of rough groups. In addition, they presented some examples and properties.\\

The main purpose of this paper is to introduce some basic definitions and results about topological rough groups and topological rough subgroups. We also introduce the certesian product of topological rough groups.

The organization of the paper is as follows:
Section $2$ gives basic results and definitions on rough groups and rough homomorphisms. In Section $3$, following results and definitions of \cite{Abdull}, we give some more interesting and nice results about topological rough groups. In Section $4$ we prove that the product of topological rough groups is a topological rough group. Further, an example is provided. Finally in Section $5$ We will introduce topological rough group homomorphisms and topological rough group homeomorphisms which are related to topological rough groups defined by Bagirmaz et al. in \cite{Abdull}. Then we present rough actions and rough homogenous spaces, and discuss some of their properties. We also define a rough kernel. For the details of topological group theory  we follow \cite{Arh}.\\

This paper is produced from the PhD thesis of Ms. Nof Alharbi registered in King Abdulaziz University.

\section{ Rough groups and rough homomorphisms}
 First, we give the definition of rough groups introduced by Biswas and Nanda in $1994$.

Let $(U,R)$ be an approximation space such that $U$ is any set and $R$ is an equivalence relation on $U$.  For a subset $X\subset U$,
$$\overline{X} = \{[x]_R : [x]_R \cap X \neq \emptyset \}$$ and
$$ \underline{X }= \{[x]_R : [x]_R \subset X \}.$$
 Suppose that $(*)$ is a binary operation defined on $U$. We will use $xy$ instead of $x*y$ for all composition of elements $x,y \in U$ as well as for composition of subsets $XY$, where $X, Y \subseteq U$.

\begin{definition}
\cite{Abdull} Let $G=(\underline{G},\overline{G})$ be a rough set in the approximation space $(U,R)$. Then $G=(\underline{G}, \overline{G})$ is called a rough group if the following conditions are satisfied:
\begin{enumerate}
\item $\forall x,y \in G, xy \in \overline{G}$ (closed);
\item $(xy)z=x(yz), \forall x,y,z \in \overline{G}$ (associative law);
\item $\forall x \in G, \exists e \in \overline{G}$ such that $xe=ex=x$ (e is the rough identity element);
\item $\forall x \in G, \exists y \in G$ such that $xy=yx=e$ ($y$ is the rough inverse element of $x$. It is denoted as $x^{-1}$).
\end{enumerate}
\end{definition}

\begin{definition}
\cite{Abdull} A non-empty rough subset $H=(\underline{H},\overline{H})$ of a rough group $G=(\underline{G},\overline{G})$ is called a rough subgroup if it is a rough group itself.
\end{definition}

The rough set $G=(\underline{G}, \overline{G})$  is a trivial rough subgroup of itself. Also the rough set $e=(\underline{e},\overline{e})$ is a trivial rough subgroup of the rough group $G$ if $e \in G$.

\begin{theorem}
\cite{Abdull}A rough subset $H$ is a rough subgroup of the rough group $G$ if the two conditions are satisfied:
\begin{enumerate}
\item $\forall x, y \in H, xy \in \overline{H}$;
\item $\forall y \in H, y^{-1} \in H$.
\end{enumerate}
\end{theorem}
Also, a rough normal subgroup can be defined. Let $N$ be a rough subgroup of the rough group $G$, then $N$ is called a rough normal subgroup of $G$ if for all $x \in G, xN=Nx$
\begin{definition}
\cite{Homo} Let $(U_{1}, R_{1})$ and $(U_{2},R_{2})$ be two approximation spaces and $*, *^{'}$ be two binary operations on $U_{1}$ and $U_{2}$, respectively. Suppose that $G_{1} \subseteq U_{1}$, $G_{2}\subseteq U_{2}$ are rough groups. If the mapping $\varphi: \overline{G_{1}} \rightarrow \overline{G_{2}}$  satisfies that for all $x, y \in \overline{G_{1}}$, $\varphi(x * y) = \varphi(x) *^{'} \varphi(y)$, then $\varphi$ is called {\bf a rough homomorphism}.
\end{definition}

\begin{definition}
\cite{Homo}  A rough homomorphism $\varphi$ from a rough group $G_{1}$ to a rough group ${G_{2}}$ is called:
\begin{enumerate}
\item a rough epimorphism (or surjective) if $\varphi: \overline{G_{1}} \rightarrow \overline{G_{2}}$ is onto.

\item a rough embedding (or monomorphism) if $\varphi: \overline{G_{1}} \rightarrow \overline{G_{2}}$ is one-to -one.

\item a rough isomorphism if $\varphi: \overline{G_{1}} \rightarrow \overline{G_{2}}$ is both onto and one-to-one.

\end{enumerate}
\end{definition}

\section{Topological rough groups}

Here, we study a topological rough group, which is an ordinary topology on a rough group, i.e., a topology $\tau$ on $\overline{G}$ induced  a subspace topology $\tau_{G}$ on $G$. Suppose that $(U,R)$ is  an approximation space with a binary operation $*$ on U. Let $G$ be a rough group in $U$.

\begin{definition}\label{rg}
\cite{Abdull} A topological rough group is a rough group $G$ with a topology $\tau$ on $\overline{G}$ satisfying the following conditions:
\begin{enumerate}
\item The product mapping $f: G \times G \rightarrow \overline{G}$ defined by $f(x,y)=xy$ is continuous with respect to a product topology on $G\times G$ and the topology $\tau_{G}$ on $G$  induced by $\tau$;
\item The inverse mapping $\iota: G \rightarrow G$ defined by $\iota(x)= x^{-1}$ is continuous with respect to the topology $\tau_{G}$ on $G$ induced by $\tau$.
\end{enumerate}
\end{definition}

Elements in the topological rough group $G$ are elements in the original rough set $G$ with ignoring elements in approximations.

\begin{example}\label{exrg1}
\cite{Abdull} Let $U= \{ \overline{0}, \overline{1}, \overline{2} \}$ be  any group with $3$ elements. Let $U / \mathcal{R} = \{ \{\overline{0}, \overline{2}\}, \{ \overline{1}\}\}$ be a classification of equivalent relation. Let $G= \{ \overline{1}, \overline{2} \}$, then $\underline{G}= \{\overline{1}\}$ and $\overline{G}= \{\overline{0}, \overline{1}, \overline{2}\}= U$. A topology on $\overline{G}$ is $\tau = \{\emptyset, \overline{G}, \{\overline{1}\},\{ \overline{2}\}, \{ \overline{1}, \overline{2}\}\},$  then the relative topology is $\tau_{G} = \{ \emptyset, G, \{\overline{1}\},\{ \overline{2}\} \}$. The two conditions in Definition $\ref{rg}$ are satisfied, hence $G$ is a topological rough group.
\end{example}
\begin{example}
Let $U=\mathbb{R}$ and  $U/\mathcal{R} =\{\{x: x\geq 0 \}, \{x: x<0\}\}$ be a partition of $\mathbb{R}$. Consider $G= {\mathbb{R}}^* = \mathbb{R}-0.$ Then $G$ is a rough group with addition. It is also a topological rough group with the usual topology on $\mathbb{R}$.
\end{example}

\begin{example}\label{exrg2}
 \cite{Abdull} Consider $U=S_{4}$ the set of all permutations of four objects. Let $(*)$ be the multiplication operation of permutations. Let
$$U / \mathcal{R} =\{ E_{1},E_{2},E_{3},E_{4}\},$$ be a classification of $U,$ where $$E_{1}=\{1,(12),(13),(14),(23),(24),(34)\}$$
$$E_{2}=\{(123),(132),(142),(124),(134),(143),(234),(243)\}$$
$$E_{3}=\{(1234),(1243),(1342),(1324),(1423),(1432)\}$$
$$E_{4}=\{(12)(34),(13)(24),(14)(23)\}.$$
Let $G=\{(12),(123),(132)\}$, then $\overline{G}= E_{1} \cup E_{2}$. Clearly, $G$ is a rough group. Consider a topology on $\overline{G}$ as  $\tau = \{\emptyset, \overline{G},\{(12)\}, \{1,(123),(132) \},\{ 1,(12),(123),(132)\}  \},$ then the relative topology on $G$ is $\tau_{G} = \{\emptyset, G, \{(12)\},\{ (123),(132) \}\}$. The two conditions in Definition $\ref{rg}$ are satisfied, hence $G$ is a topological rough group.

\end{example}

\begin{proposition}\label{La}
\cite{Abdull} Let $G$ be a topological rough group and fix $a \in G$. Then
\begin{enumerate}
\item the mapping $L_{a}: G \rightarrow \overline{G}$ defined by $L_{a}(x)=ax$, is one-to-one and continuous for all $x \in G$.
\item the mapping $R_{a}: G \rightarrow \overline{G}$ defined by $R_{a}(x)=xa$, is one-to-one and continuous for all $x \in G$.
\item the inverse mapping $\iota: G \rightarrow G$ is a homeomorphism for all $x \in G$.
\end{enumerate}
\end{proposition}

\begin{proposition}\label{G^{-1}}
\cite{Abdull} Let $G$ be a  topological rough group, then $G= G^{-1}$.

\end{proposition}

Now, we will define rough symmetry in topological rough group.
\begin{definition}
Let $G$ be a topological rough group. Then a subset $U$ of $G$ is called rough symmetric if $U=U^{-1}$.
\end{definition}
\begin{corollary}
Every rough subgroup of topological rough group is rough symmetric.
\end{corollary}
\begin{proof}
Let $G$ be a topological rough group and let $H$ be a rough subgroup of $G.$ Then $H$ is a topological rough subgroup with relative topology. By Proposition $\ref{G^{-1}}$, it is clear that $H=H^{-1}$. Hence $H$ is rough symmetric.

\end{proof}
\begin{proposition}\label{V^{-1}}
\cite{Abdull} Let $G$ be a topological rough group and $V \subseteq G$. Then $V$ is open(closed) $\iff V^{-1}$ is open(closed).
\end{proposition}

\begin{proposition}\label{W}
\cite{Abdull} Let $G$ be a topological rough group and $W$be an open set in $\overline{G}$ with $e \in W$. Then there exists an open set $V$ with $e \in V$ such that $V=V^{-1}$ and $VV \subseteq W$.
\end{proposition}

\begin{proposition}\label{topgp}
\cite{Abdull} Let $G$ be a topological rough group. If $G = \overline{G}$ then $G$ is a topological group.

\end{proposition}

From the defintion of rough subgroup, we obtain the following result.

\begin{theorem}
Let $G$ be a topological rough group. Then closure of any rough symetric subset $A$ of $G$ is again rough symetric.
\end{theorem}

\begin{proof}
The inverse map $\iota: G \rightarrow G$ is a homeomorphism, then $cl(A)=(cl(A))^{-1}.$

\end{proof}

\begin{theorem}
Let $G$ be a  topological rough group, and let $H$ be a rough subgroup. If $cl(H)$ in $\overline{G}$ is subset of $G,$ then $cl(H)$ is a rough subgroup in $G.$
\end{theorem}
\begin{proof}
\begin{enumerate}
\item Closed under product: Let $x,y \in cl(H) \implies xy \in \overline{G} \implies \exists$ open set $U \in \overline{G}$ such that $xy \in U.$ Claim prove that $U \wedge H\neq \emptyset.$ Consider the multiplication map $\mu: G \times G \rightarrow \overline{G} \implies \exists$ open sets $W,V$ of $G$ such that $x \in W, y \in V,$ then $W \wedge H \neq \emptyset, V \wedge H \neq \emptyset.$ Since topology on $G$ is relative topology on $\overline{G},$ then there exists open sets $W^{'},V^{'}$ of $\overline{G}$ such that $W \subseteq W^{'}, V \subseteq V^{'},$ implies $W^{'} \wedge H \neq \emptyset, V^{'}\wedge H \neq \emptyset.$ Then $\mu(W \times V) \wedge \overline{H} \neq \emptyset$, but we have $\mu (W \times V)\subseteq U, \implies \overline{H} \wedge U \neq \emptyset$. Hence $xy \in cl(H) \subseteq \overline{cl(H)}.$



\item Inverse element; since the inverse map is a homeomorphism, ${cl(H)}^{-1}=cl(H).$

\end{enumerate}
\end{proof}




\section{Cartesian Product of Topological Rough Groups}

Let $(U,\mathcal{R}_{1})$ and $(V,\mathcal{R}_{2})$ be  approximation spaces with binary operations $*_{1}$ and $*_{2}$, respectively. Consider the cartesian product of $U$ and $V$:  let  $x,x^{'} \in U$ and $y,y^{'} \in V,$ then $(x,y), (x^{'},y^{'}) \in U \times V$. Define $*$ as $(x,y) * (x^{'},y^{'})= (x *_{1} x^{'}, y *_{2} y^{'}),$ then $*$ is a binary operation on $U \times V$. From our paper \cite{AAO1}, we have that the product of equivalence relations is also an equivalence relation on $U \times V$.

\begin{theorem}
 Let $G_{1} \subseteq U$ and $G_{2} \subseteq V$ be two rough groups. Then the cartesian product $G_{1} \times G_{2}$ is also a rough group.
 \end{theorem}
The following conditions are satisfied:
\begin{enumerate}
\item For all $(x ,y),(x^{'},y^{'}) \in G_{1} \times G_{2},$ $(x_{1},y_{1}^{'}) * (x_{2},y_{2}^{'})=(x_{1}*_{1}x_{2},y_{1}^{'}*_{2} y_{2}^{'}) \in \overline{G_{1}} \times \overline{G_{2}}$.
\item Associative law is satisfied over all elements in $\overline{G_{1}} \times \overline{G_{2}}$.
\item There exists an identity element $(e,e^{'}) \in \overline{G_{1}} \times \overline{G_{2}}$ such that $\forall (x,x^{'}) \in G_{1} \times G_{2}, (x,x^{'})\times (e,e^{'}) =(e,e^{'})\times (x,x^{'})=(ex,e' x^{'})= (x. x')$.
\item For all $(x,x^{'})\in G_{1} \times G_{2}$, there  exists an element $(y,y^{'}) \in G_{1} \times G_{2}$ such that $(x,x^{'}) * (y,y^{'})=(y,y^{'}) * (x,x^{'})=(e,e^{'}).$
\end{enumerate}

\begin{example}
Consider Example $\ref{exrg1}$ where $U= \{ \overline{0}, \overline{1},\overline{2}\}$ and  $U/R= \{\{\overline{0}, \overline{2}\},\{ \overline{1} \}\}.$ Then the cartesian product $U \times U$ is as  follows: $$U\times U=\{ (\overline{0},\overline{0}),(\overline{0},\overline{2}),(\overline{0},\overline{1}),(\overline{2},\overline{0}),(\overline{2},\overline{2}),(\overline{2},\overline{1}),(\overline{1},\overline{0}),(\overline{1},\overline{2}),(\overline{1},\overline{1})\},$$ then the new classification is $$\{  \{ \overline{0},\overline{0}),(\overline{0},\overline{2}),(\overline{2},\overline{0}),(\overline{2},\overline{2}) \}, \{  (\overline{0},\overline{1}), (\overline{2},\overline{1})\}, \{ (\overline{1},\overline{0}),(\overline{1},\overline{2}) \}, \{(\overline{1},\overline{1})\}  \}. $$
 Consider the rough group $G=\{\overline{1}, \overline{2}\},$ then the cartesian product $G \times G$ is $$G \times G = \{ (\overline{2},\overline{2}), (\overline{2},\overline{1}), (\overline{1},\overline{2}), (\overline{1},\overline{1}) \},$$ where $\overline{G \times G}= \overline{G} \times \overline{G}= U \times U.$ From the definition of a rough group, we have that
\begin{enumerate}
\item the multiplication of elements in $G \times G$ is closed under $\overline{G} \times \overline{G},$ i.e. $(\overline{2},\overline{2}) (\overline{2},\overline{2})= (\overline{1},\overline{1}), (\overline{2},\overline{2}) (\overline{2},\overline{1}) = (\overline{1},\overline{0}), (\overline{2},\overline{2})(\overline{1},\overline{1})= (\overline{0},\overline{0}), (\overline{2},\overline{2}) (\overline{1},\overline{2})=(\overline{0},\overline{1}) , (\overline{2},\overline{1})(\overline{2},\overline{1})= (\overline{1},\overline{2}), \\ (\overline{2},\overline{1}) (\overline{1},\overline{1})= (\overline{0},\overline{2}),(\overline{2},\overline{1})(\overline{
1},\overline{2})= (\overline{0},\overline{0}),(\overline{1},\overline{1})(\overline{1},\overline{1})= (\overline{2},\overline{2}) , (\overline{1},\overline{1})(\overline{1},\overline{2})  =  (\overline{2},\overline{0}).$

\item There exists $(\overline{0},\overline{0}) \in  \overline{G} \times \overline{G}$ such that for every $(g,g^{'}) \in G \times G,$ we have $(\overline{0},\overline{0})(g,g^{'})= (g,g^{'}).$
 \item For every element of $G \times G,$ there exists an inverse element in $G \times G,$ where  $(\overline{1},\overline{1})^{-1}= (\overline{2},\overline{2}) \in G \times G,$ $(\overline{2},\overline{1})^{-1}=(\overline{1},\overline{2})\in G \times G.$
\item The associative law is satisfied.
\end{enumerate}
Hence $G \times G$ is  a rough group.\\

From Example $\ref{exrg1},$ we have $\tau = \{\emptyset, \overline{G}, \{\overline{1}\},\{ \overline{2}\},\{ \overline{1},\overline{2}\}\}$ as a topology on $\overline{G},$ then $\tau \times \tau$ is the product topology of $\overline{G}\times \overline{G}.$ Also we have $\tau_{G} = \{\emptyset, G, \{\overline{1}\},\{ \overline{2}\}\}$ as a relative topology on $G$, then $\tau_{G} \times \tau_{G}$ is a product topology on $G \times G$ induced by $\tau \times \tau.$\\

Consider the multiplication map $\mu: (G \times G) \times (G \times G ) \rightarrow \overline{G} \times \overline{G}$. This map is continuous with respect to topology $\tau \times \tau$ and the product topology on $(G \times G) \times (G \times G ).$ Also, consider the inverse map $\iota: G \times G \rightarrow G \times G$ is continuous. Hence $G \times G$ is a topological rough group.

\end{example}

\section{Rough Action and Rough Homeogenous Spaces in Classical Set Topology }
Note that, we are interested in rough action and rough homegenous spaces in classical set topology using rough set groups. Let $(U,R_{1}),(V.R_{2})$ be approximation spaces such that $G_{1} \subseteq U$ and $G_{2} \subseteq V$. Let $G_{1}$, $G_{2}$ be topological rough groups such that $\tau_{{1}}, \tau_{{2}}$ are topologies on $\overline{G_{1}},\overline{G_{2}}$ respectively inducing $\tau_{G_{1}}, \tau_{G_{2}}$ on $G_{1},G_{2}$ respectively.\\

 A mapping $f: \overline{G_{1}} \rightarrow \overline{G_{2}}$ is called {\bf a topological rough group homomorphism}, if $f$ is a rough homomorphism and continuous with respect to the topology $\tau_{2}$ on $\overline{G_{2}}$ inducing $\tau_{G_{2}}$ on $G_{2}$ and a topology $\tau_{1}$ on $\overline{G_{1}}$ inducing $\tau_{G_{1}}$ on $G_{1}.$\\

 Topological rough group homomorphism $f: \overline{G_{1}} \rightarrow \overline{G_{2}}$ is called {\bf a topological rough group homeomorphism}, if there exists a topological rough group homomorphism $f^{-1}$ such that $f^{-1} \circ f= 1_{G_{1}}$.\\

Let $(U,R)$ be an approximation space. Assume that, $G$ and $X$ are two subsets of $U$ such that $G$ is a topological rough group, and $\overline{X}$ is a topological rough space inducing the topology rough space $X$ i.e, rough set with ordinary topology. Then we are ready to give the definition of the action of a rough group $G$ on a rough space.
\begin{definition}\label{action}
A continuous map $\mu: \overline{G} \times \overline{X} \rightarrow \overline{X}$ (resp. $\mu: \overline{X} \times \overline{G} \rightarrow \overline{X}$) is called a left (resp. right) rough action of $G$ on $X$, if it satisfies the following conditions:
\begin{enumerate}
\item\label{con1} $g(g^{'}x)= (gg^{'})x$ (resp. $((xg)g^{'}=x(gg^{'}))$, for every $g,g^{'} \in \overline{G}$ and $x \in \overline{X}.$

\item \label{con2} $ex =x$ (resp. $xe=x$), $e \in \overline{G},$ for every $x \in \overline{X}.$
\end{enumerate}
Then $X$ is called a rough $G$-space.

\end{definition}
 The action $\mu$ is said to be {\bf effective} if $gx=g^{'}x$, for every $x \in \overline{X}$ implies $g=g^{'}$. In addition, the action $\mu$ is said to be {\bf transitive}, if for every  $x, x^{'} \in \overline{X}$, there is $g \in \overline{G}$ such that $gx =x^{'}$.

\begin{definition}
Let $X$ be a rough  $G$-space. Then $X$ is said to be {\bf topologically rough homogeneous} if for any $x,y \in \overline{X}$, there is a topological  homeomorphism $\Phi: \overline{X} \rightarrow \overline{X}$ such that $\Phi(x) = y$.
\end{definition}
 \begin{theorem}
Let $G$ be a topological rough group and $X$ a rough $G$- space. Then left(resp. right) transformation map $L_{g}(R_{g}): \overline{X} \rightarrow \overline{X}$, for every $g \in G,$ which is defined by $L_{g}(x)=gx (R_{g}(x)= xg)$, is a topological homeomorphism.
\end{theorem}

\begin{proof}
Indeed, the continuity of the action $\mu$ implies the continuity of $L_{g}$. The conditions \ref{con1} and \ref{con2} in Defintion \ref{action} are respectively equivalent to 
\begin{enumerate}
\item $L_{g} \circ L_{g^{'}} = L_{gg^{'}} $.
\item $L_{e} = 1_{X}$.
\end{enumerate}
Therefore, the maps $L_{g}$ and $L_{g^{-1}}$ are inverses of each other. Thus, $L_{g}$ is a topologically homeomorphism from $\overline{X}$ to $\overline{X}$.
\end{proof}
Note that, the left (resp. right) transformation map $L_{g}(R_{g}): \overline{X} \rightarrow \overline{X},$ is not topological homeomorphism for every $g \in \overline{G},$  only in the case that $\overline{G}$ is a group.
\begin{remark}
For every open set $O$ in  $\overline{X}$, and $g \in G$ $L_{g}(O) = gO$ is open in $\overline{X}$.

\end{remark}


\begin{remark}
Let $H$ be a rough subgroup of $G,$ and let $\overline{G}$ be a group. Then $\overline{H}$ acts on $G,$ and thus $G$ is a homogeneous space.

\end{remark}

\begin{remark}

If $\overline{G}$ is a group, then $G$ acts roughly on itself.

\end{remark}
\begin{theorem}
Let $G$ be a topological rough group such that $\overline{G}$ is a group. For any open subset $U$ of $\overline{G},$ if  $A$ is a subset of $\overline{G},$ then $AU$ (respectively $UA$) is open in $\overline{G}$.
\end{theorem}

\begin{proof}
The fact that $\overline{G}$ is a group implies that $G$ acts on itself, then for every $g \in \overline{G}, L_{g},$ or $(R_{g})$ is a topological homeomorphism. The rest of proof follows immediately from left (right) transformation, because that $AU= \cup_{a \in A} L_{a}(U)$ and $UA= \cup_{a \in A}R_{a}(U).$

\end{proof}
\begin{theorem}
Let $G$ be an topological rough group such that $\overline{G}$ is a group. Let $H$ be a rough subgroup of $G$ such that $\overline{H}$ is closed under multiplication. If there is an open set $W$  of $G$ such that $e \in W$ and $W \subseteq H$, then $\overline{H}$ is open set in $\overline{G}$.
\end{theorem}
\begin{proof}
Let $W$ be a non-empty open set of $G$ such that $W \subseteq H$ and $e \in W$. Then for every $h \in\overline{ H}$, $L_{h}(W)= h W$ is open in $\overline{G}$. Hence  $\overline{H}=\bigcup_{h \in \overline{H}} h W$ is open in $\overline{G}$.
\end{proof}

\begin{theorem}
Let $G$ be a topological rough group such that $\overline{G}$ is a group and let $H$ be a rough subgroup of $G$. Let $W$ be an open set in $G$ such that $W \subseteq H$. Then for every $h \in H$,  $hW$ is an open set in $\overline{H}$.
\end{theorem}

\begin{proof}
Since $\overline{H} \subseteq \overline{G},$ and $\overline{G}$ is a group, then $L_{h}$ is a topological homeomorphism. By left transformation $L_{h}(W)=hW$ is open in $\overline{G}.$ The fact that $W \subseteq H,$ implies $hW \subseteq \overline{H}.$ Hence $hW$ is open in $\overline{H}.$

\end{proof}

\begin{definition}
Let G be a topological rough group and let $\mathcal{B} \subseteq \tau_{G}$ be a base for $\tau_{G}.$ For $g \in \overline{G},$ the family $$\mathcal{B}_{g}=\{ O \in \mathcal{B}: g \in O\} \subseteq \mathcal{B}$$ is called a {\bf base at} $g.$ 
\end{definition}
 \begin{theorem}
Let $G$ be a topological rough group such that the identity element $e \in G$ and $\overline{G}$ is closed under multiplication. Let $G$ be an open set in $\overline{G},$ for $g \in G$ the base of $g$ in $\overline{G}$ is equal to $$\mathcal{B}_{g}=\{gO: O \in \mathcal{B}_{e}\},$$ where $\mathcal{B}_{e}$ is the base of the identity $e$ in $\tau_{G}.$
\end{theorem}

\begin{proof}
Since $g \in G,$ then $g \in \overline{G},$ let $U$ be an open set in $\overline{G}$ and let $g \in U.$ Since $e \in G,$ and $G$ is a topological rough group, this implies that, there are two open sets $W,V$ such that $g \in W, e \in V$ and $\mu(W \times V) \subseteq U.$
We have $G$ is an open set in $\tau,$ then $V$ is a neibourhood of $e$ in $\tau.$ Then there is a basic open set $O \in \mathcal{B}_{e}$ such that $e \in O \subseteq V.$ Hence $L_{g}(O)= gO \subseteq \mu(W \times O) \subseteq \mu(W \times V) \subseteq U.$

\end{proof}

\space

\begin{definition}\label{ker}
Let $\Phi : \overline{G_{1}} \rightarrow \overline{G_{2}}$ be a topological rough group homomorphism and let $e_{2}$ be the rough identity element in $G_{2}$. Then 
\begin{align*}
ker(\Phi) = \{ g \in {G_{1}}: \Phi(g) = e_{2}\}.
\end{align*}
is called the {\bf rough kernel} associated to the map $\Phi$.
\end{definition}
\begin{remark}
Our Definition \ref{ker} is equivelent to definition of rough kernel in rough groups in \cite{Homo}.

\end{remark}

\begin{theorem}
Let $\Phi$ be rough homomorphism from $G_{1}$ to $G_{2}.$ Then the rough kernel is a rough normal subgroup of $G_{1}.$
\end{theorem}
\begin{proof}
For every $x,y \in ker(\Phi),$ we have $\Phi(x)=e_{2},$ and $\Phi(y)= e_{2}.$
\begin{enumerate}
\item Since $\Phi(x*y)=\Phi(x)*^{'}\Phi(y)=e_{2},$ we have $x*y \in ker(\Phi).$ 
\item We have $\Phi(x^{-1})=(\Phi(x))^{-1} = (e_{2})^{-1}.$ Hence $ker(\Phi)$ is a rough subgroup of $G_{1}.$
\item For every $x \in G_{1}$ and $r \in ker(\Phi),$ we have $\Phi(x*r*x^{-1})= \Phi(x)*^{'}  \Phi(r)*^{'} \Phi(x^{-1})= e_{2}.$ Therefore, $x*r*x^{-1} \in ker(\Phi).$ so that $ker(\Phi)$ is a rough normal subgroup of $G_{1}.$

\end{enumerate}
\end{proof}

\begin{remark}
The rough kernel is always a subset of upper approximation of $G_{1}$. Indeed, if $\overline{G_{1}}$  is a group then the kernel is a normal subgroup of $\overline{G_{1}}.$

\end{remark}

\begin{example}
Consider the map $\Phi: \overline{G_{U}} \rightarrow \overline{G_{S_{4}}}$, where $G_{U}$ and $G_{S_{4}}$ are rough groups in Example $\ref{exrg1}$ and Example $\ref{exrg2}$ respectivily. Define $f$ as follow:
$$\Phi (\overline{0})=(1), \Phi (\overline{1})=(1), \Phi(\overline{2})=(1).$$ Clearly $\Phi$ is continuous and homomorphism. Hence $\Phi$ is a topological rough group homomorphism. From Definition \ref{ker} it is easy to see that $ker(\Phi)=\{  \overline{1}, \overline{2}\}$ which is a subset of $\overline{G_{U}}.$ Moreover, $ker(\Phi)$ is a rough normal subgroup of $G_{U}$.
\end{example}

{\bf{Acknowledgement}} The authors wish to thank the Deanship for Scientific Research (DSR) at King Abdulaziz University for financially funding this project under grant no. KEP-PhD-2-130-39.

\vspace*{1cm}

\end{document}